\documentclass[twoside,a4paper,reqno,11pt]{amsart}
\usepackage{amsfonts, amsbsy, amsmath, amssymb, latexsym}
\usepackage{mathrsfs,array}
\usepackage[top=30mm,right=30mm,bottom=30mm,left=30mm]{geometry}
\usepackage{stmaryrd}
\usepackage{bm}

\usepackage{hyperref}
\def\d{\delta}

\newcommand{\N}{\mathbb{N}}
\newcommand{\R}{\mathbb{R}}
\newcommand{\e}{\epsilon}
 
\renewcommand{\l}{\lambda} \renewcommand{\O}{\Omega}

 \renewcommand{\to}{\rightarrow}

\newcommand{\leqs}{\leqslant}
\newcommand{\geqs}{\geqslant}

 \newcommand{\vs}{\vspace{3mm}}

\def\no{\noindent}
\newcommand{\normeq}{\lhd}

\makeatletter
\newcommand{\imod}[1]{\allowbreak\mkern4mu({\operator@font mod}\,\,#1)}
\makeatother

\newtheorem{theorem}{Theorem}
\newtheorem*{conj*}{Conjecture}

\newtheorem{propn}[theorem]{Proposition}

\newtheorem{corol}[theorem]{Corollary}

\newtheorem{thm}{Theorem}[section]
\newtheorem{prop}[thm]{Proposition}
\newtheorem{lem}[thm]{Lemma}
\newtheorem{cor}[thm]{Corollary}

\theoremstyle{definition}
\newtheorem{rem}[thm]{Remark}

\begin{document}

\author{Timothy C. Burness}
 \address{T.C. Burness, School of Mathematics, University of Bristol, Bristol BS8 1TW, UK}
 \email{t.burness@bristol.ac.uk}

 \author{Martin W. Liebeck}
\address{M.W. Liebeck, Department of Mathematics,
    Imperial College, London SW7 2BZ, UK}
\email{m.liebeck@imperial.ac.uk}

\author{Aner Shalev}
\address{A. Shalev, Institute of Mathematics, Hebrew University, Jerusalem 91904, Israel}
\email{shalev@math.huji.ac.il}

\title{The depth of a finite simple group}

\begin{abstract}
We introduce the notion of the depth of a finite group $G$, defined as the minimal length of an unrefinable chain of subgroups from $G$ to the trivial subgroup. In this paper we investigate the depth of (non-abelian) finite simple groups. We determine the simple groups of minimal depth, and show, somewhat surprisingly, that alternating groups have bounded depth. We also establish general upper bounds on the depth of simple groups of Lie type, and study the relation between the depth and the much studied notion of the length of simple groups. The proofs of
our main theorems depend (among other tools) on a deep number-theoretic result, namely, Helfgott's recent solution of the ternary Goldbach conjecture.
\end{abstract}

\footnotetext{
The first and third authors acknowledge the hospitality and support of Imperial College, London, while
part of this work was carried out.
The third author acknowledges the support of ISF grant 1117/13 and the Vinik chair of mathematics which he holds.}

\date{\today}
\maketitle

\section{Introduction}\label{s:intro}

An {\it unrefinable} chain of length $t$ of a finite group $G$ is a chain of subgroups
\begin{equation}\label{e:chain}
G = G_0 > G_1 > \cdots > G_{t-1} > G_t=1,
\end{equation}
where each $G_{i}$ is a maximal subgroup of $G_{i-1}$. We define the \emph{depth} of $G$, denoted by $\l(G)$, to be the minimal length of an unrefinable chain. For example, if $G$ is a cyclic group of order $n \geqs 2$, then $\l(G) = \O(n)$, the number of prime divisors of $n$ (counting multiplicities). In particular, $\l(G)=1$ if and only if $G$ has prime order.

In this paper we are interested in the depth of finite simple groups (by which we mean non-abelian finite simple
groups). For such a group $G$, it is easy to show that $\l(G) \geqs 3$ (see Corollary \ref{d3}). In fact, this lower bound is best possible, and our first theorem determines the simple groups of minimal depth.

\begin{theorem}\label{t:main1}
Let $G$ be a finite simple group. Then $\l(G)= 3$ if and only if $G$ is one of the groups recorded in Table \ref{tab:main}.
\end{theorem}

\renewcommand{\arraystretch}{1.06}
\begin{table}[h]
$$\begin{array}{lll} \hline\hline
G && \hspace{5mm} \mbox{Conditions} \\ \hline
A_p && \hspace{5mm} \mbox{$p$ and $(p-1)/2$ prime, $p \not\in \{7,11,23\}$} \\
{\rm L}_2(q) && \left\{\begin{array}{l}
\mbox{$(q+1)/(2,q-1)$ or $(q-1)/(2,q-1)$ prime, $q \ne 9$; or} \\
\mbox{$q$ prime and $q \equiv \pm 3, \pm 13 \imod{40}$; or} \\
\mbox{$q=3^k$ with $k \geqs 3$ prime}
\end{array}\right. \\
{\rm L}_n^{\e}(q) && \hspace{5mm} \mbox{$n$ and $\frac{q^n-\e}{(q-\e)\,(n,q-\e)}$ both prime, $n \geqs 3$ and} \\
&& \hspace{5mm} (n,q,\e) \ne (3,4,+), (3,3,-), (3,5,-), (5,2,-) \\
{}^2B_2(q) & & \hspace{5mm} \mbox{$q-1$ prime} \\
{\rm M}_{23},  \; \mathbb{B} && \\ \hline \hline
\end{array}$$
\caption{The simple groups $G$ with $\l(G)=3$}
\label{tab:main}
\end{table}
\renewcommand{\arraystretch}{1}

In particular, there are infinitely many simple groups with depth $3$.

Next we turn our attention to upper bounds. Firstly, using Helfgott's solution of the ternary Goldbach conjecture (see \cite{H}, as well as Vinogradov's classical result \cite{Vin} for sufficiently large numbers), we show that alternating groups have bounded depth.

\begin{theorem}\label{t:main2}
We have $\l(A_n) \leqs 23$ for all $n$.
\end{theorem}

This is in stark contrast to the situation for groups of Lie type (see Proposition \ref{p:l2q} for the exact depth of ${\rm L}_{2}(p^k)$ for a prime $p$ and odd integer $k$).

\begin{theorem}\label{t:main3}
For any $n \in \mathbb{N}$, there exists a prime power $q$ such that $\l({\rm L}_{2}(q))>n$.
\end{theorem}

Next, applying Theorem \ref{t:main2} above and other tools, we establish a general upper bound on the depth of finite simple groups of Lie type.

\begin{theorem}\label{t:main4}
Let $G = G(q)$ be a simple group of Lie type, where $q=p^k$ for a prime $p$. Then either
\[
\l(G) \leqs 3\O(k) + 36,
\]
or one of the following holds:
\begin{itemize}\addtolength{\itemsep}{0.2\baselineskip}
\item[{\rm (i)}] $G = {\rm L}_{2}(2^k)$ or ${}^2B_2(2^k)$ and
\[
\l(G) \leqs \O(k) + 1+\min\{\O(2^r - 1) : r \in \pi(k)\},
\]
where $\pi(k)$ is the set of prime divisors of $k$.
\item[{\rm (ii)}] $G = {\rm U}_{n}(2^k)$, $n$ is odd, $k$ is even and
\[
\l(G) \leqs 3\O(k)+2\O(2^{2^a}+1)+35,
\]
where $k=2^ab$ with $b$ odd.
\end{itemize}
\end{theorem}

Note that Proposition \ref{p:suz} determines the precise depth of the groups in case (i) in Theorem \ref{t:main4}.
Also Proposition \ref{p:l2q} gives the depth of ${\rm L}_2(p^k)$ for $k$ odd.
A detailed investigation of the depth of simple groups of Lie type will be presented  in a forthcoming paper.

Define a function $f_1: \N \to \R$ by
\[
f_1(k) = 3 \log_2 k + 2k/\log_2(2k) + 35.
\]
Applying Theorem \ref{t:main4} with some elementary number theory we obtain the following.

\begin{corol}\label{c:main5}
With the above notation we have $$\l(G(p^k)) < f_1(k).$$
\end{corol}

The depths of the sporadic simple groups are routine to compute, and are given in Lemma \ref{spor}.

The \emph{length} $l(G)$ of a finite group $G$ is defined to be the maximal length of a strictly descending chain of subgroups from $G$ to $1$. The length of simple groups has been the subject of numerous papers since the 1960s (see \cite{AST,Bab,CST, Har,Jan,SST, ST, ST2}, for example).

What are the relations between the depth $\l(G)$ and the length $l(G)$ of a finite (or a finite simple)
group $G$? Clearly, $\l(G) \leqs l(G)$.  By a well known theorem of Iwasawa \cite{I}, $\l(G)= l(G)$ (namely, all unrefinable
chains in $G$ have the same length) if and only if $G$ is supersolvable. In particular, $\l(G) < l(G)$ if
$G$ is simple. Note that there are families of finite simple groups $G$ for which $\l(G)$ is bounded while $l(G)$ is unbounded.
For example, $l(A_n)$ is of the order of $\frac{3}{2}n$ by \cite{CST}, whereas $\l(A_n) \leqs 23$ by Theorem \ref{t:main2}.
We show below that a similar phenomenon occurs even for simple groups of minimal depth.

\begin{theorem}\label{t:main6}
For any $n \in \mathbb{N}$, there exists a finite simple group $G$ of minimal depth $\l(G) = 3$
such that $l(G) > n$. In fact, we may take $G = {\rm L}_{2}(p)$ for a suitable prime $p$.
\end{theorem}

Next, we show that $\l(G)$ is always asymptotically much smaller than $l(G)$.
We need some notation. For integers $l \geqs 36$ define
\[
h(l) = \max\left\{\log_2 (l-2) + \frac{l-2}{\log_2 (l-2)}+1, \; 3 \log_2 ((l-4)/3) +
\frac{2(l-4)}{3\log_2(2(l-4)/3)} + 35 \right\}.
\]
Define a function $f_2: \N \to \R$ by $f_2(l) = l$ for $l < 36$ and
\[
f_2(l) = \min\{l, h(l)\},
\]
for $l \geqs 36$.

\begin{theorem}\label{t:main7}
Let $G$ be a finite simple group. Then
\[
\l(G) \leqs f_2(l(G)).
\]
In particular, $\l(G) \leqs (1+o(1))\frac{l(G)}{\log_2 l(G)}$.
\end{theorem}

We also obtain better upper bounds on $\l(G)$ -- see Theorem \ref{t:3.8}.

As for lower bounds, we show the following.

\begin{propn}\label{p:main8}
There exist infinitely many finite simple groups $G_j$ $(j \geqs 1)$ satisfying $l(G_j) \to \infty$
and $\l(G_j) > \log_3 l(G_j) + 1$.
\end{propn}

It would be nice to close the gap between the upper bound in Theorem \ref{t:main7} and the
lower bound in Proposition \ref{p:main8}. However, this depends on formidable open problems
in Number Theory. See the discussion at the end of Section \ref{s:proofs}.

In \cite{BWZ}, the expression $l(G)-\l(G)$ is called the \emph{chain difference} of $G$, denoted by ${\rm cd}(G)$. It
follows from Iwasawa's theorem mentioned above that ${\rm cd}(G) \geqs 1$ for all finite simple groups $G$. Using the classification theorem, the simple groups $G$ with ${\rm cd}(G)=1$ were determined by Brewster et al. \cite{BWZ} -- the only examples are $A_6$ and ${\rm L}_{2}(p)$ for certain primes $p$ (it is not known whether there are infinitely many examples). In \cite{HS}, Hartenstein and Solomon present a more elementary proof of the same result, by means of a reduction to groups with dihedral or semi-dihedral Sylow $2$-subgroups. In particular, the proof in \cite{HS} does not require the classification of finite simple groups.

The finite simple groups of minimal length $4$ have depth $3$ and chain difference $1$, and so can be read
off from Theorem \ref{t:main1} above, together with \cite{BWZ}. The precise list is given in \cite[Theorem 3.2]{Pet}.
On the other hand, our results imply that the chain difference of a finite simple group is usually large.

In fact, using Theorem \ref{t:main7} it follows immediately that the length $l(G)$ of a finite simple group $G$ is bounded above in terms of its chain difference ${\rm cd}(G) = l(G) - \l(G)$, and even in terms of its \emph{chain ratio}, defined by ${\rm cr}(G) = l(G)/\l(G)$.

\begin{corol}
\label{c:main9}
We have
\[
l(G) \leqs (1+o(1)) {\rm cd}(G)
\]
and
\[
l(G) \leqs 2^{(1+o(1)){\rm cr}(G)},
\]
for all finite simple groups $G$, where $o(1)$ is $o_{{\rm cd}(G)}(1)$ and $o_{{\rm cr}(G)}(1)$ respectively.

In particular, the following statements are equivalent for any collection $\mathcal{S}$ of finite simple groups:
\begin{itemize}\addtolength{\itemsep}{0.2\baselineskip}
\item[{\rm (i)}] The set $\{ {\rm cr}(G) \, : \, G \in \mathcal{S} \}$ is bounded.
\item[{\rm (ii)}] The set $\{ {\rm cd}(G) \, : \, G \in \mathcal{S} \}$ is bounded.
\item[{\rm (iii)}] The set $\{ l(G) \, : \, G \in \mathcal{S} \}$ is bounded.
\end{itemize}
\end{corol}

Indeed, the first two assertions of Corollary \ref{c:main9} (which imply the third one) follow from the last statement of Theorem \ref{t:main7}.
We note that condition (iii) above is equivalent (for any collection $\mathcal{S}$ of finite groups $G$) to a purely number theoretic condition.
Indeed, it is trivial that $l(G) \leqs \O(|G|)$, and by \cite[Proposition 2.2]{AST}  we have $\O(|G|) \leqs l(G)^2$. Thus the set $\{ l(G) \, : \, G \in \mathcal{S} \}$ is bounded if and only if the set $\{ \O(|G|) \, : \, G \in \mathcal{S} \}$ is bounded. Furthermore, it is known that there are infinitely many finite simple groups of bounded length; indeed \cite[Corollary D]{AST} implies that there are infinitely many primes
$p$ with $l({\rm L}_{2}(p)) \leqs 20$.

Our study of the depth of finite simple groups is partly motivated by our recent work
on the minimal and random generation of so-called \emph{$t$-maximal subgroups} of finite simple groups, where $t=1,2,3$ (see \cite{BLS1, BLS2}).

The proofs of results \ref{t:main1}--\ref{p:main8} are given in Section \ref{s:proofs} and we record some relevant preliminary results in Section \ref{s:prel}. In this paper we adopt the notation from \cite{KL} for simple groups of Lie type. In particular  we write ${\rm PSL}_{n}(q) = {\rm L}_{n}(q) = {\rm L}_{n}^{+}(q)$ and ${\rm PSU}_{n}(q) = {\rm U}_{n}(q) = {\rm L}_{n}^{-}(q)$, etc. We are grateful to Roger Heath-Brown for helpful correspondence.

\section{Preliminaries}\label{s:prel}

We begin with elementary observations.

\begin{lem}\label{norm}
Let $G$ be a finite group and let $\mathcal{M}$ be the set of maximal subgroups of $G$.
\begin{itemize}\addtolength{\itemsep}{0.2\baselineskip}
\item[{\rm (i)}] $\l(G) = 1 + \min\{\l(M) \, :\, M \in \mathcal{M}\}$.
\item[{\rm (ii)}] If $N$ is a normal subgroup of $G$, then
\[
\l(G/N) \leqs \l(G) \leqs \l(G/N)+\l(N).
\]
\end{itemize}
\end{lem}

\begin{lem}\label{frob}
Suppose $\l(G)=2$ and let $M$ be a maximal subgroup of $G$ of prime order. Then either $M\normeq G$, or $G$ is a Frobenius group of the form $NM$, where $N \normeq G$ and $M$ acts fixed point freely on $N$.
\end{lem}

\begin{proof}
If $M$ is not normal in $G$, then the action of $G$ on the cosets of $M$ is Frobenius. \end{proof}

\begin{cor}\label{d3}
If $G$ is a finite simple group, then $\l(G) \geqs 3$.
\end{cor}

\begin{lem}\label{soph}
Suppose $G$ is a finite simple group, and $M$ is a nilpotent maximal subgroup of $G$. Then $M$ is a non-abelian Sylow $2$-subgroup of $G$.
\end{lem}

\begin{proof}
Suppose first that $M$ has a nontrivial Sylow $p$-subgroup $P$ for some odd prime $p$. Then $M = N_G(P)$ since $M$ is maximal, and hence also $M = N_G(Z(J(P)))$, where $J(P)$ is the Thompson subgroup of $P$. Therefore $G$ has a normal $p$-complement by the Glauberman-Thompson normal $p$-complement theorem (see \cite[Section 8.3]{Gor}, for example). This is a contradiction.

Hence $M$ is a 2-group. Also $M \in {\rm Syl}_2(G)$ since $M = N_G(M)$. Finally, if $M$ is abelian then $M = Z(M) = Z(N_G(M))$, and so $G$ has a normal 2-complement by Burnside's normal $p$-complement theorem. Hence $M$ is non-abelian.
\end{proof}

\begin{rem}
There are genuine examples in Lemma \ref{soph}. For instance, $D_{16}$ is a maximal subgroup of ${\rm L}_{2}(17)$.
\end{rem}

Our final result in this section concerns the existence of alternating (or symmetric) maximal subgroups of certain simple classical groups. For the proof, we need to recall a standard construction.

Let $p$ be a prime, let $d \geqs 5$ be an integer and consider the permutation module $\mathbb{F}_{p}^d$ for the symmetric group $S_d$. Define subspaces
\begin{equation}\label{e:fd}
U=\{(a_1, \ldots, a_d) \,:\, \sum_{i}a_i=0\},\;\; W=\{(a,\ldots, a) \,:\, a \in \mathbb{F}_{p}\}
\end{equation}
of $\mathbb{F}_{p}^d$, and observe that $U$ and $W$ are the only nonzero proper $A_d$-invariant submodules of $\mathbb{F}_{p}^d$. Then
$V=U/(U \cap W)$
is the \emph{fully deleted permutation module} for $A_d$, which is an absolutely irreducible $A_d$-module over $\mathbb{F}_{p}$. Set $n = \dim V$ and note that
$n = d-2$ if $p$ divides $d$, otherwise $n = d-1$.

If $p$ is odd, then the corresponding representation embeds $A_d$ into an orthogonal group $\O_{n}^{\e}(p)$.
If $p=2$ then $n$ is even and either $d \equiv 2 \imod{4}$ and $A_d$ embeds in ${\rm Sp}_{n}(2)$, or $d \not\equiv 2 \imod{4}$ and we obtain an embedding $A_d \leqs \Omega_{n}^{\e}(2)$ (see \cite[p.187]{KL} for further details).

\begin{lem}\label{max}
Let $G = \O_{n}^{\e}(p)$, where $n \geqs 5$, $p$ is a prime and one of the following holds:
\begin{itemize}\addtolength{\itemsep}{0.2\baselineskip}
\item[{\rm (i)}] $np$ is odd, $n\ne 7$ and $(n+1,p)=1$;
\item[{\rm (ii)}] $(p,\e)=(2,+)$ and $n \equiv 0,6 \imod{8}$;
\item[{\rm (iii)}] $(p,\e)=(2,-)$ and $n \equiv 2,4 \imod{8}$.
\end{itemize}
Then $G$ has a maximal alternating or symmetric subgroup. The same conclusion holds if $G = {\rm Sp}_{n}(2)$, $n \geqs 8$ and $n \equiv 0 \imod{4}$.
\end{lem}

\begin{proof}
For $n \leqs 12$, we refer the reader to the relevant tables in \cite{BHR}. Now assume $n>12$. Let $V$ be the natural module for $G$.

Suppose (i) holds and define $\d \in \{1,2\}$ to be 2 if $p$ divides $n+2$, and 1 otherwise. Consider the embedding of $A_{n+\d}$ in $G=\O_n(p) = \O(V)$ afforded by the fully deleted permutation module for $A_{n+\d}$ over $\mathbb{F}_{p}$. Set $H = N_G(A_{n+\d}) = A_{n+\d}$ or $S_{n+\d}$.

We claim that $H$ is a maximal subgroup of $G$. To see this, suppose there is a subgroup $K$ of $G$ such that $H<K<G$. Since $K$ is irreducible and the $(-1)$-eigenspace of $(1,2)(3,4) \in H$ on $V$ is $2$-dimensional, the possibilities for $K$ are given in \cite[Theorem 7.1]{GS}. However, by inspection we see that no examples arise with $n>12$, whence $H$ is maximal.  (Note that $H$ is clearly primitive and tensor-indecomposable on $V$, so \cite{GS} applies.)

A very similar argument applies in cases (ii) and (iii). For example, consider (iii). Here $G = \O_{n}^{-}(2)$ and $n \equiv 2,4 \imod{8}$. Set $H = A_{n+\delta}$, where $\delta=2$ if $n \equiv 2 \imod{8}$ and $\delta=1$ if $n \equiv 4 \imod{8}$. As before, the fully deleted permutation module $V = V_n(2)$ embeds $H$ in $G$ (note that transpositions in $S_{n+\delta}$ act as transvections on $V$, so $S_{n+\delta} \not\leqs G$). As before, we can establish the maximality of $H$ by applying \cite[Theorem 7.1]{GS}, noting that $(1,2)(3,4) \in H$ has Jordan form $[J_2^2,J_{1}^{n-4}]$ on $V$. An entirely similar argument shows that $G = {\rm Sp}_{n}(2)$ (with $n \geqs 8$ and $n \equiv 0 \imod{4}$) has a maximal subgroup $S_{n+2}$.
\end{proof}

\section{Proofs}\label{s:proofs}

Let $G$ be a finite group. Define a \emph{$t$-chain} of $G$ to be an unrefinable chain of subgroups of length $t$ as in \eqref{e:chain}.

\begin{lem}\label{lap1}
If $p$ is prime, then $\l({\rm L}_2(p)) \leqs 4$.
\end{lem}

\begin{proof}
The result is clear for $p\leqs 3$. And for $p\geqs 5$, ${\rm L}_2(p)$ has a maximal subgroup isomorphic to $A_4$, $S_4$ or $A_5$ (see \cite{Dix}), and it is easy to check that all of these groups have depth at most $3$.
\end{proof}

\begin{cor}\label{c:1}
If $p$ is prime, then $\l(A_{p+1}) \leqs 5$.
\end{cor}

\begin{proof}
Again, the claim is clear if $p \leqs 3$, so assume $p \geqs 5$. If $p \not\in \{7,11,23\}$ then ${\rm L}_{2}(p)$ is a maximal subgroup of $A_{p+1}$ (see \cite{LPS}), so in these cases the result follows immediately from Lemma \ref{lap1}. For $p \in \{7,11,23\}$ it is easy to check that $\l(A_{p+1}) = 5$. For example,
\[
A_{24}>{\rm M}_{24}>{\rm M}_{23} > 23{:}11> 11 > 1
\]
is a $5$-chain.
\end{proof}

\begin{lem}\label{spor}
The depth of each sporadic simple group $G$ is given in Table \ref{tab:spor}. In particular, $\l(G) \leqs 6$, with equality if and only if $G = {\rm He}$.
\end{lem}

\begin{proof}
This is easily checked by inspecting the list of maximal subgroups in \cite{At}.
\end{proof}

\renewcommand{\arraystretch}{1.1}
\begin{table}[h]
\[
\begin{array}{lccccccccccccc} \hline\hline
G & {\rm M}_{11} & {\rm M}_{12} & {\rm M}_{22} & {\rm M}_{23} & {\rm M}_{24} & {\rm J}_{1} & {\rm J}_{2} & {\rm J}_{3} & {\rm J}_{4} & {\rm HS} & {\rm Suz} & {\rm McL} & {\rm Ru}  \\ \hline
\l(G) & 4 & 4 & 4 & 3 & 4 & 4 & 4 & 5 & 4 & 5 & 4 & 5 & 5  \\ \hline\hline
& & & & & & & & & & & & &   \\ \hline\hline
& {\rm He} & {\rm Ly} & {\rm O'N} & {\rm Co}_{1} & {\rm Co}_{2} & {\rm Co}_{3} & {\rm Fi}_{22} & {\rm Fi}_{23} & {\rm Fi}_{24}' & {\rm HN} & {\rm Th} & \mathbb{B} & \mathbb{M}   \\ \hline
& 6 & 4 & 5 & 5 & 4 & 4 & 5 & 4 & 4 & 5 & 4 & 3 & 4    \\ \hline\hline
\end{array}
\]
\caption{The depth of sporadic simple groups}
\label{tab:spor}
\end{table}
\renewcommand{\arraystretch}{1}

We are now in a position to prove our main theorems.

\subsection{Proof of Theorem \ref{t:main1}}

Let $G = G_0>G_1>G_2>G_3=1$ be a $3$-chain, so each $G_i$ is maximal in $G_{i-1}$.
Then $G_2$ has prime order $r$, say, and by Lemma \ref{frob}, either $G_1$ is Frobenius or $G_2\normeq  G_1$.

If $G_1$ has odd order, then it is given by \cite[Theorem 1]{LSax} and the relevant cases are recorded in Table \ref{tab:main}. Now assume $|G_1|$ is even.

Suppose $G_1 = NG_2 = N.r$ is Frobenius. As $G_2$ is maximal in $G_1$, $N$ is elementary abelian and thus one of the following holds:
\begin{itemize}\addtolength{\itemsep}{0.2\baselineskip}
\item[(a)] $N = 2^k$ and $G_2=r$ acts fixed point freely on $N$;
\item[(b)] $|N| = s$ is prime, $r=2$ and $G_1$ is dihedral.
\end{itemize}
The finite simple groups $G$ with a maximal subgroup $G_1$ of the form $2^k.r$ or $D_{2s}$ can be determined by inspection of  \cite{KL} (for classical groups), \cite{CLSS} (for exceptional groups of Lie type), \cite{At} (for sporadic groups), and is elementary for alternating groups. The examples are listed in Table \ref{tab:mm} and they also appear in Table \ref{tab:main}.

\renewcommand{\arraystretch}{1.1}
\begin{table}[h]
\[
\begin{array}{lll} \hline\hline
G & G_1 & \hspace{5mm} \mbox{Conditions} \\ \hline
{\rm L}_2(2^k) & 2^k.(2^k-1) & \hspace{5mm} \mbox{$2^k-1$ prime} \\
                    & D_{2(2^k \pm 1)} &  \hspace{5mm} \mbox{$2^k \pm 1$ prime} \\
{\rm L}_2(q) & D_{q \pm 1} & \hspace{5mm} \mbox{$(q \pm 1)/2$ prime, $q \ne 9$} \\
 & A_4 & \left\{\begin{array}{l}
\mbox{$q$ prime and either $q=5$ or $q \equiv \pm 3, \pm 13 \imod{40}$; or} \\
\mbox{$q=3^a$ with $a \geqs 3$ prime}
\end{array}\right. \\
{}^2B_2(q) & D_{2(q-1)} & \hspace{5mm} \mbox{$q-1$ prime} \\ \hline\hline
\end{array}
\]
\caption{The simple groups $G$ with a maximal subgroup $G_1$ of the form $2^k.r$ or $D_{2r}$, with $r$ prime}
\label{tab:mm}
\end{table}
\renewcommand{\arraystretch}{1}

Finally, let us assume $G_2\normeq G_1$, so $G_1/G_2$ has prime order $t$, say. Then $G_1$ is non-abelian by Lemma \ref{soph}. Since $|G_1|$ is even, it follows that $t=2$ and $G_1 = D_{2r}$ is dihedral. This case was dealt with in (b) above. \hspace *{\fill}$\Box$

\vs

By combining Theorem \ref{t:main1} and Lemma \ref{lap1}, we obtain the following corollary.

\begin{cor}\label{c:l2p}
If $p$ is an odd prime, then
$$\l({\rm L}_{2}(p)) = \left\{\begin{array}{ll}
2 & \mbox{$p=3$} \\
3 & \mbox{$p \geqs 5$ and either $(p-1)/2$ or $(p+1)/2$ prime,} \\
& \mbox{or $p \equiv \pm 3, \pm 13 \imod{40}$}  \\
4 & \mbox{otherwise.}
\end{array}\right.$$
\end{cor}

This can be extended as follows.

\begin{prop}\label{p:l2q}
Let $p$ be a prime and let $k \geqs 1$ be an odd integer. Suppose $(p,k) \ne (2,1)$ and let $\pi(k)$ be the set of prime divisors of $k$. Then
$$\l({\rm L}_{2}(p^k)) = \left\{\begin{array}{ll}
\Omega(k)+ 1 + \min\{\Omega(2^r \pm 1) \,:\, r \in \pi(k)\} & \mbox{if $p=2$} \\
\Omega(k)+ \l({\rm L}_{2}(p)) & \mbox{if $p \geqs 3$.}
\end{array}\right.$$
\end{prop}

\begin{proof}
First assume that $p$ is odd. The proof goes by induction on $k$, the case $k=1$ being trivial.
Now suppose $k>1$ and let $G = {\rm L}_2(p^k)$. By \cite{Dix}, the maximal subgroups of $G$ are as follows:
\begin{equation}\label{maxsub}
p^k.((p^k-1)/2),\; D_{p^k \pm 1},\; {\rm L}_{2}(p^{k/s}),
\end{equation}
where $s$ is a prime divisor of $k$, and it is easy to see that
\[
\l(p^k.((p^k-1)/2)) = \l(D_{p^k-1}) = \O(p^k-1),\;\; \l(D_{p^k+1}) = \O(p^k+1).
\]
By induction, $\l({\rm L}_2(p^{k/s})) = \O(k)-1+\l({\rm L}_2(p))$.
Since $\O(p^k\pm 1) \geqs \O(p \pm 1) + \O(k)$, it follows from Corollary \ref{c:l2p} that among the maximal subgroups in (\ref{maxsub}), ${\rm L}_2(p^{k/s})$ has minimal depth. Hence
\[
\l({\rm L}_2(p^k)) = 1+\l({\rm L}_2(p^{k/s})) = \O(k)+\l({\rm L}_2(p)),
\]
and the proof is complete.

Now assume $p=2$. This time we induct on $\O(k)$. For the base case $\O(k)=1$, $k$ is prime and the maximal subgroups of
${\rm L}_2(2^k)$ are
\begin{equation}\label{maxsub2}
2^k.(2^k-1),\; D_{2(2^k\pm 1)}.
\end{equation}
We have
$\l(2^k.(2^k-1)) = \l(D_{2(2^k-1)}) = \O(2^k-1)+1$ and $\l(D_{2(2^k+1)}) = \O(2^k+1)+1$, and
the conclusion follows for $k$ prime. For $k$ non-prime (i.e.  $\O(k)>1$), the maximal subgroups of ${\rm L}_2(p^k)$ are as in \eqref{maxsub2}, together with ${\rm L}_2(2^{k/s})$ for $s \in \pi(k)$, and an induction argument very similar to the one for $p$ odd gives the conclusion.
\end{proof}

\begin{rem}
A similar result can be established for $\l({\rm L}_{2}(p^k))$ when $k$ is even, but the details are  more complicated (see Proposition \ref{p:suz} for the case $p=2$).
\end{rem}

\subsection{Proof of Proposition \ref{p:main8}}

The proof combines Proposition \ref{p:l2q} above with \cite[Theorem A]{SST}.
The latter result shows that, for a finite simple Lie type group $G_r(p^k)$ of rank $r$
with a Borel subgroup $B$ we have $l(G_r(p^k)) = r + l(B)$ provided $k \geqs F(p,r)$.

For $i \geqs 1$ let $H_i = {\rm L}_{2}(3^{3^i})$ and let $B_i < H_i$ be a Borel subgroup.
It follows from the above mentioned result that, for some constant $c > 0$ we have
\[
l(H_i) = 1 + l(B_i)
\]
for all $i > c$. Now, let $P_i < B_i$ be a Sylow $3$-subgroup of $H_i$.
Since $B_i$ is solvable we have
\[
l(B_i) = \O(|B_i|) = \O((3^{3^i}-1)/2) + \O(|P_i|) = \O((3^{3^i}-1)/2) + 3^i.
\]
Note that $(3^{3^i}-1)/2$ is not divisible
by primes less than $7$. Hence $\O((3^{3^i}-1)/2) \leqs \log_7 ((3^{3^i}-1)/2) < 3^i \log_7{3}$.
This yields
\[
l(H_i) < 1 + 3^i(1 + \log_7{3}) = 1 + 3^i \log_7 {21} <3^{i+1}
\]
for all $i > c$.

Next, Proposition \ref{p:l2q} shows that
\[
\l(H_i) = \O(3^i) + \l({\rm L}_{2}(3)) = i+2.
\]
Hence, for $i > c$, we have
\[
\l(H_i) = i + 2 > \log_3 l(H_i) + 1.
\]
Setting $G_j = H_{j+c}$ for $j \geqs 1$, we complete the proof.
\hspace *{\fill}$\Box$

\subsection{Proof of Theorem \ref{t:main2}}

Let $G=A_n$. If $n \leqs 10$ then it is easy to check that $\l(G) \leqs 5$, so let us assume $n \geqs 11$.
By Vinogradov's Theorem \cite{Vin}, every sufficiently large odd integer $n$ is the sum of three primes, and this has recently been extended to all odd $n \geqs 7$ by Helfgott \cite{H}. Set $\d = 1$ or 0 according as $n$ is odd or even, and choose primes $p_1,p_2,p_3$ such that
\[
n-3-\d = p_1+p_2+p_3,
\]
so
\[
A:= A_{p_1+1}\times A_{p_2+1}\times A_{p_3+1}  < A_{p_1+p_2+p_3+3} = A_{n-\d} \leqs  G.
\]

We claim that there is an unrefinable chain of length at most $8$ from $G$ to $A$. To see this, first observe that the stabilizer in $A_d$ of a $k$-element subset of $\{1, \ldots, d\}$ (with $2 \leqs k \leqs d/2$) is a subgroup of the form $(A_k \times A_{d-k}).2$. Moreover, if $k \ne d/2$ then this is a  maximal subgroup by \cite{LPS}, so there is an unrefinable chain of length $2$ from $A_d$ to $A_k \times A_{d-k}$. If $k=d/2$, then there is one of length $3$, namely
\[
A_d > (A_{d/2} \times A_{d/2}).2^2 > (A_{d/2} \times A_{d/2}).2 > A_{d/2} \times A_{d/2}.
\]
Now, if $n-\delta \ne 2(p_1+1)$ and $p_2 \ne p_3$, then
\[
A_{n-\d} > (A_{p_1+1} \times A_{n-\d-p_1-1}).2 > A_{p_1+1} \times A_{n-\d-p_1-1}  > A_{p_1+1} \times (A_{p_2+1} \times A_{p_3+1}).2 > A
\]
is an unrefinable chain of length $4$. Since $A_n > S_{n-1} > A_{n-1}$ is unrefinable, it follows that there is an unrefinable chain of length at most $6$ from $G$ to $A$. Similarly, if either $n-\delta = 2(p_1+1)$ or $p_2 = p_3$, then we can find a chain of length at most $8$. This justifies the claim.

Finally, since $\l(A_{p_i+1}) \leqs 5$ by Corollary \ref{c:1}, we conclude that
\[
\l(G) \leqs 8 + 3\cdot 5 = 23\]
and the proof of Theorem \ref{t:main2} is complete. \hspace *{\fill}$\Box$

\subsection{Proof of Theorem \ref{t:main3}}

Let $n$ be a positive integer and let $p_1,\ldots, p_n$ be distinct odd primes. Set $k = p_1\cdots p_n$. Then Proposition \ref{p:l2q} gives $\l({\rm L}_{2}(2^k)) \geqs \O(k)+2 = n+2$ and the result follows.  \hspace *{\fill}$\Box$

\subsection{Proof of Theorem \ref{t:main4}}

Let $G=G(q)$ be a finite simple group of Lie type over $\mathbb{F}_{q}$, where $q=p^k$ for a prime $p$. To begin with, let us assume that $G$ is not one of the following:
\begin{itemize}\addtolength{\itemsep}{0.2\baselineskip}
\item[(a)] ${\rm L}_{2}(2^k)$ with $k \geqs 2$;
\item[(b)] ${}^2B_2(2^k)$ with $k \geqs 3$ odd;
\item[(c)] ${\rm U}_{n}(2^k)$ with $n$ odd and $k$ even.
\end{itemize}
We will handle these special cases at the end of the proof.

In the following, unless stated otherwise, the assertions concerning the unrefinability of chains follow from the maximality results in \cite{BHR, KL} for classical groups and \cite{LSS} for exceptional groups. Our goal is to verify the bound
\begin{equation}\label{e:bd}
\l(G) \leqs 3\O(k)+36.
\end{equation}

\vs

\no \emph{Case 1. Untwisted groups.}

\vs

First assume $G = G(q)$ is of untwisted type (excluding (a) above). For any prime divisor $r$ of $k$, $G(q)$ has a maximal subfield subgroup of the form $G(q^{k/r}).[\d]$, where $\d \in \{1,r,2r\}$ (see \cite[Theorem 1]{BGL}). We deduce that there is an unrefinable chain of length at most $3\O(k)$ from $G$ to $G(p)$, and hence
\begin{equation}\label{trivi}
\l(G) \leqs 3\O(k)+\l(G(p)).
\end{equation}

We now consider the possibilities for $G(p)$. First assume $G(p) = \O_n(p)$, with $np$ odd and $n \geqs 7$. If $n\ne 7$ and $(n+1,p)=1$ then Lemma \ref{max} implies that $G(p)$ has a maximal alternating or symmetric subgroup, in which case $\l(G) \leqs 3\O(k)+25$ by Theorem \ref{t:main2}. Now assume $p$ divides $n+1$. Then
\[
\O_n(p)>\O_{n-1}^{+}(p).2 > \O_{n-1}^{+}(p) > \O_{n-2}(p).2 > \O_{n-2}(p)
\]
is an unrefinable chain of length $4$. Moreover, $(n-1,p)=1$ so $\O_{n-2}(p)$ has a maximal alternating or symmetric subgroup. This gives $\l(G) \leqs 3\O(k)+4+25$ as required. Finally, for $n=7$ there is an unrefinable chain
$\O_7(p) > {\rm Sp}_6(2) > S_8$, and the conclusion follows easily.

Next assume $G(p) = {\rm P\O}_{2n}^{+}(p)$, where $n \geqs 4$ and $p$ is odd. Then $G(p)$ has a maximal subgroup of the form $\O_{n-1}(p).r$ with $r \in \{1,2\}$, so by applying the bound in the previous paragraph we get $\l(G) \leqs 3\O(k)+29+2$.

Now suppose $G(p) = {\rm Sp}_{2n}(2)'$. It is easy to check that the groups ${\rm Sp}_{4}(2)' \cong A_6$ and ${\rm Sp}_{6}(2)$ have depth $4$ and $5$, respectively, so we may assume $n\geqs 4$. If $n$ is even then Lemma \ref{max} implies that $G(p)$ has a maximal symmetric subgroup. On the other hand, if $n$ is odd then
\[
{\rm Sp}_{2n}(2) > {\rm Sp}_{2n-2}(2) \times {\rm Sp}_{2}(2) > {\rm Sp}_{2n-2}(2) \times 3 > {\rm Sp}_{2n-2}(2)
\]
is an unrefinable chain and ${\rm Sp}_{2n-2}(2)$ has a maximal symmetric subgroup (again, by Lemma \ref{max}). In both cases, we conclude that $\l(G) \leqs 3\O(k)+3+25$, so \eqref{e:bd} holds. Moreover, for $G(p) = \O_{2n}^{+}(2)$ we get $\l(G) \leqs 3\O(k)+29$ because ${\rm Sp}_{2n-2}(2)$ is a maximal subgroup of $G(p)$.

Next consider $G(p) = {\rm PSp}_{2n}(p)$ with $p$ odd and $n\geqs 2$. Here $G(p)$ has a maximal imprimitive subgroup $M = ({\rm Sp}_2(p) \wr S_n)/Z$, where $Z = Z({\rm Sp}_{2n}(p)) = \{\pm I_{2n}\}$.

First we claim that there is an unrefinable chain
\begin{equation}\label{unref}
{\rm Sp}_2(p) \wr S_n  = M_0 > M_1> \cdots > M_s = C_6 \wr S_n
\end{equation}
of length $s \leqs 3$. If $p \equiv \pm 1 \imod{10}$, then $2.A_5$ is a maximal subgroup of ${\rm Sp}_{2}(p)$ and we can take
\begin{equation}\label{e:chain2}
{\rm Sp}_2(p) \wr S_n  > (2.A_5) \wr S_n  > (2.A_4) \wr S_n  > C_6 \wr S_n.
\end{equation}
To see that this is unrefinable, consider a subgroup $K$ such that
\[
H=C_6 \wr S_n < K \leqs L=(2.A_4) \wr S_n.
\]
Then $K \cap (2.A_4)^n \leqs (2.A_4)^n$ is a subdirect product containing $(C_6)^n$, so $C_6 \leqs K \cap L_i \normeq L_i$, where $L_i$ is the $i$-th copy of $2.A_4$ in the direct product $(2.A_4)^n$. Therefore $K \cap L_i = L_i$, so $K$ contains $(2.A_4)^n$
and thus $K=L$. A similar argument establishes the maximality of the other inclusions in \eqref{e:chain2} and we omit the details. If $p \not\equiv \pm 1 \imod{10}$ then either $2.S_4$ or $2.A_4$ is maximal in ${\rm Sp}_{2}(p)$ and the details are very similar. This establishes the claim (\ref{unref}).

Finally, we claim that there is an unrefinable chain
\[
C_6 \wr S_n = H_0 > H_1 > \cdots > H_t = 2.S_n = Z.S_n
\]
of length $t \leqs 5$. For example, if $n \equiv 0 \imod{6}$ then
\[
C_6 \wr S_n > 3^{n-1}.2^n.S_n > 3.2^n.S_n> C_2 \wr S_n > 2^{n-1}.S_n > 2.S_n
\]
is an unrefinable chain of length $5$. Here we are using the fact that the only proper nontrivial $S_n$-invariant subgroups of $r^n$ ($r$ prime) are $U \cong r^{n-1}$ and $W \cong C_r$ (note that $U$ and $W$ are the subspaces in \eqref{e:fd}, setting $p=r$). Similarly, there is a chain of length $5$ if $n \equiv \pm 2,3 \imod{6}$, and one of length $4$ if $n \equiv \pm 1 \imod{6}$. We deduce that there is an unrefinable chain of length at most $8$ from $G(p)$ to $S_n$, whence $\l(G) \leqs 3\O(k)+8+24$ by Theorem \ref{t:main2}.

To complete the proof for untwisted classical groups, suppose $G(p) = {\rm L}_n(p)$. The case $n=2$ follows from Lemma \ref{lap1}, so assume $n\geqs 3$. If $n$ is even then $G(p)$ has a maximal subgroup $M={\rm PSp}_{n}(p).r$ with $r \in \{1,2\}$ and our earlier work shows that $\l(M) \leqs 33$. Now assume $n$ is odd. If $p$ is odd then $G(p)$ has a maximal subgroup $M={\rm PSO}_{n}(p) = \O_n(p).2$ and the result follows since $\l(M) \leqs 30$ as above. Finally, suppose $n$ is odd and $p=2$. In this case, there is an unrefinable chain
\[
G(2) = {\rm SL}_n(2)  > 2^{n-1}.{\rm SL}_{n-1}(2) > {\rm SL}_{n-1}(2)
\]
and so the previous argument gives $\l(G) \leqs 3\O(k)+36$.

Now suppose $G(p)$ is of exceptional Lie type. In each case, we can choose a maximal subgroup $M$ as follows (see \cite{LSS}):
\[
\begin{array}{llllll}
\hline
G(p) & E_8(p) & E_7(p) & E_6(p) & F_4(p) & G_2(p) \\
\hline
M & d.{\rm P\O}_{16}^{+}(p).d & {\rm L}_{2}(p^7).[7d] & F_4(p) & d.\O_9(p) & {\rm SL}_{3}(p).2 \\
\hline
\end{array}
\]
where $d=(2,p-1)$. In each case, the desired bound quickly follows from our above analysis of untwisted classical groups. For example, if $G(p)=E_8(p)$ then
\[
\l(G(p)) \leqs 3+\l({\rm P\O}_{16}^{+}(p)) \leqs 3+31.
\]
Similarly, suppose $G(p) = E_7(p)$ and $M = {\rm L}_{2}(p^7).[7d]$. If $p=2$ then
\[
E_7(2) > {\rm L}_{2}(2^7).7 > {\rm L}_{2}(2^7) > D_{2(2^7-1)} > C_{2^7-1} > 1
\]
is an unrefinable chain. For odd $p$, there is an unrefinable chain from $E_7(p)$ to ${\rm L}_{2}(p)$ of length $4$, and Lemma \ref{lap1} implies that $\l({\rm L}_{2}(p)) \leqs 4$. The other cases are similar and we omit the details.

\vs

\no \emph{Case 2. Twisted groups.}

\vs

Now let us consider the twisted groups of Lie type, excluding the cases labelled (b) and (c) above. First assume $G = {}^2G_2(3^k)$ with $k$ odd. Taking a chain of subfield subgroups of length $\O(k)$, we can get down to ${}^2G_2(3) \cong {\rm L}_{2}(8).3$. The latter has depth $4$, so $\l(G) \leqs \O(k)+4$. Similarly, $\l({}^2F_4(2^k)') \leqs \O(k)+5$.

In each of the remaining cases, the goal is to find a short unrefinable chain from $G$ to a simple untwisted group of Lie type $H$, and then apply the bounds in Case 1.

Suppose $G = {\rm U}_{n}(q)$ is a unitary group. If $n \geqs 4$ is even then there is an unrefinable chain of length at most $2$ from $G$ to $H = {\rm PSp}_{n}(q)$, so $\l(G) \leqs 2+3\O(k)+32$. Similarly, if $nq$ is odd then $H = \O_{n}(q)$ and the same bound holds. If $n$ is odd and $q=2^k$ with $k$ odd then we can use maximal subfield subgroups to find an unrefinable chain of length at most $3\O(k)$ from $G$ to $H={\rm U}_{n}(2)$. If $n=3$ then $\l(H) = 4$, so we can assume $n \geqs 5$. Now $H$ has a maximal subgroup $a.{\rm U}_{n-1}(2).b$, where $a=3/(3,n)$ and $b=(3,n-1)$, so $\l(H) \leqs \l({\rm U}_{n-1}(2))+2 \leqs 36$ as above and thus $\l(G) \leqs 3\O(k)+36$.

For $G = {\rm P\O}_{2n}^{-}(q)$ with $q$ odd, there is an unrefinable chain of length at most $2$ from $G$ to $\O_{n-1}(q)$ and the result quickly follows. The case $G = \O_{2n}^{-}(q)$ with $q$ even is also easy since ${\rm Sp}_{2n-2}(q)$ is a maximal subgroup.

Finally, if $G = {}^2E_6(q)$ or ${}^3D_4(q)$, then $G$ has a maximal subgroup $F_4(q)$ or $G_2(q)$, respectively, and the result follows from the bounds on $\l(F_4(q))$ and $\l(G_2(q))$ in Case 1.

\vs

\no \emph{Case 3. The remaining cases.}

\vs

To complete the proof, we may assume that one of the following holds:
\begin{itemize}\addtolength{\itemsep}{0.2\baselineskip}
\item[(a)] $G = {\rm L}_{2}(2^k)$ with $k \geqs 2$;
\item[(b)] $G = {}^2B_2(2^k)$ with $k \geqs 3$ odd;
\item[(c)] $G = {\rm U}_{n}(2^k)$ with $n$ odd and $k$ even.
\end{itemize}

First suppose $G = G(2^k)$ is of type ${\rm L}_{2}(2^k)$ or ${}^2B_2(2^k)$. Let $\pi(k)$ be the set of prime divisors of $k$. For any $r \in \pi(k)$, there is an unrefinable chain of subfield subgroups of length $\O(k)-1$ from $G$ to $G(2^r)$. Now $G(2^r)$ has a maximal subgroup $H = D_{2(2^r-1)}$ and $\l(H) \leqs 1+\O(2^r-1)$, so
\[
\l(G) \leqs \O(k) + 1+\min\{\O(2^r - 1) : r \in \pi(k)\}
\]
as required (see Proposition \ref{p:suz} below for the exact depth of $G$ in these two cases).

Finally, let us turn to case (c), so $G = G(2^k) = {\rm U}_{n}(2^k)$ with $n$ odd and $k$ even. Write $k = 2^ab$, where $a \geqs 1$ and $b$ is odd, so $\O(k) = a+\O(b)$. By considering subfield subgroups, there is an unrefinable chain of length at most $3\O(b)$ from $G$ to $G(2^{2^a})$. Now $G(2^{2^a})$ has a maximal reducible subgroup $H = c.{\rm PGU}_{n-1}(2^{2^a})$ where $c$ divides $2^{2^a}+1$, so
\[
\l(H) \leqs 2\O(2^{2^a}+1)+\l({\rm U}_{n-1}(2^{2^a})) \leqs 2\O(2^{2^a}+1)+34+3\O(2^a)
\]
and thus
\[
\l(G) \leqs 3\O(b)+2\O(2^{2^a}+1)+35+3a \leqs 3\O(k)+2\O(2^{2^a}+1)+35.
\]

\vs

This completes the proof of Theorem \ref{t:main4}.
\hspace *{\fill}$\Box$

In fact, we can determine the exact depth of $G$ in cases (a) and (b) above.

\begin{prop}\label{p:suz}
We have
$$\l({\rm L}_{2}(2^k)) = \left\{\begin{array}{ll}
\Omega(k)+ 1 + \min\{\Omega(2^r \pm 1) \,:\, r \in \pi(k)\} & \mbox{$k \geqs 3$ odd} \\
\Omega(k)+ 2 + \min\{\Omega(2^{2^c} \pm 1) - c \,:\, 1 \leqs c \leqs a \} & \mbox{$k=2^ab$ even, $b$ odd}
\end{array}\right.$$
and
$$\l({}^2B_2(2^k)) = \O(k)+1+\min\{\O(2^r-1),\O(2^r \pm \sqrt{2^{r+1}} + 1)+1 \,:\, r \in \pi(k)\},$$
where $\pi(k)$ is the set of prime divisors of $k$.
\end{prop}

\begin{proof}
First assume $G = {\rm L}_{2}(2^k)$. In view of Proposition \ref{p:l2q}, we may assume $k=2^ab$ is even and $b \geqs 1$ is odd. By arguing as in the proof of Proposition \ref{p:l2q}, we deduce that $\l(G) = \Omega(b) + \l(H)$, where $H = {\rm L}_{2}(2^{2^a})$. Consider a $t$-chain
$$H  = H_0 > H_1 > H_2 > \cdots > H_t = 1$$
and let $s \geqs 0$ be maximal so that $H_s = {\rm L}_{2}(2^{2^c})$ is a subfield subgroup of $H$. Then $c \geqs 1$ and $\l(H) = s + \l(H_s)$. Moreover, $s = \Omega(2^{a-c})=a-c$ and the maximality of $s$ implies that $\l(H_s) = 2 + \min\{\Omega(2^{2^c} \pm 1)\}$, so
$$\l(H) = a-c + 2 + \min\{\Omega(2^{2^c} \pm 1)\}.$$
The result now follows since $\Omega(k) = a+\Omega(b)$.

Now assume $G = {}^2B_2(2^k)$, where $k \geqs 3$ is odd. Set $q=2^k$ and let $H$ be a maximal subgroup of $G$. By \cite{Suz}, $H$ is one of
$$q^{1+1}{:}(q-1),\; D_{2(q-1)},\; (q \pm \sqrt{2q}+1){:}4,\; {}^2B_2(q_0),$$
where $q_0=2^{k/s}$ for a proper prime divisor $s$ of $k$. We have
$$\l(q^{1+1}{:}(q-1)) = \O(q-1)+2,\; \l(D_{2(q-1)}) = \O(q-1)+1$$
and
$$\l((q \pm \sqrt{2q}+1){:}4) = \l(D_{2(q \pm \sqrt{2q}+1)}) +1 = \O(q \pm \sqrt{2q}+1)+2.$$
Similarly,
$$\l({}^2B_2(q_0)) \leqs \l(D_{2(q_0-1)})+1 = \O(q_0-1)+2 \leqs \O(q-1)+1$$
and
$$\l({}^2B_2(q_0)) \leqs \O(q_0 \pm \sqrt{2q_0}+1) +3 \leqs \O(q \mp \sqrt{2q}+1)+2$$
(note that $q_0 \pm \sqrt{2q_0}+1$ divides $q \mp \sqrt{2q}+1$).
Therefore, we can construct an unrefinable chain for $G$ of minimal length by descending via a sequence of $\O(k)-1$ subfield subgroups to ${}^2B_2(2^r)$ for some prime divisor $r$ of $k$. It follows that
$$\l(G) = \O(k)-1 + \min\{\O(2^r-1)+2, \l(2^r \pm \sqrt{2^{r+1}}+1)+3\,:\, r \in \pi(k)\}$$
as required.
\end{proof}

\subsection{Proof of Corollary \ref{c:main5}}

We apply Theorem \ref{t:main4}. Trivially, we have $\O(k) \leqs \log_2{k}$.
So $\l(G) \leqs 3 \O(k) + 36$ implies $\l(G) \leqs 3 \log_2{k} + 36 \leqs f_1(k)$, as required.

Next, suppose conclusion (i) of Theorem \ref{t:main4} holds, namely
$$\l(G) \leqs \O(k) +1+ \min\{\O(2^r - 1) : r \in \pi(k)\}.$$
For each prime divisor $r$ of $k$, each prime $s$ dividing $2^r-1$ satisfies $s \equiv 1 \imod{r}$,
so $s \geqs r+1$.
Hence
$$\O(2^r-1) \leqs \log_2 (2^r-1)/\log_2 (r+1) < r/\log_2 r \leqs k/\log_2 k.$$ This yields
\[
\l(G) \leqs \log_2 k + k/\log_2 k + 1
\]
and the result follows.

Finally, suppose conclusion (ii) of Theorem \ref{t:main4} holds, namely
$$\l(G) \leqs 3\O(k)+2\O(2^{2^a}+1)+35,$$ where $k = 2^a b$ with $a \geqs 1$ and $b$ odd.

Let $s$ be a prime divisor of $2^{2^a}+1$. We claim that
$s \equiv 1 \imod{2^{a+1}}$.
Indeed, let $m$ be the multiplicative order of $2$ modulo $s$. Since $2^{2^a} \equiv -1 \imod s$
we have $2^{2^{a+1}} \equiv 1 \imod s$, so $m$ divides $2^{a+1}$. But $m$ does not divide $2^a$,
hence $m = 2^{a+1}$, so $2^{a+1}$ divides $s-1$, as claimed. Therefore, $s \geqs 2^{a+1} + 1$ and thus
\[
\O(2^{2^a+1} + 1) \leqs \log_2 (2^{2^a}+1) / \log_2 (2^{a+1}+1) < 2^a/(a+1).
\]
This implies that
\begin{align*}
\l(G) < 3\O(k) + 2^{a+1}/(a+1) + 35 & \leqs 3 \log_2 k + (2k/b)/(\log_2(2k/b)) + 35 \\
& \leqs 3 \log_2 k + 2k/\log_2(2k) + 35,
\end{align*}
completing the proof. \hspace *{\fill}$\Box$

\subsection{Proof of Theorem \ref{t:main6}}

Fix $n \geqs 2$ and let $p_1, \ldots, p_{n-1}$ be the first $n-1$ primes which are greater than $5$.
Let $S$ be the set of primes $p$ satisfying  $p \equiv \pm 3, \pm 13 \imod{40}$
and $p \equiv 1 \imod{p_i}$ for $i=1, \ldots, n-1$. By the Chinese Remainder theorem and Dirichlet's theorem, $S$ is infinite.

Let $G = {\rm L}_2(p)$ with $p \in S$.
Then $\l(G) = 3$ by Theorem \ref{t:main1}. On the other hand, the dihedral group $D_{p-1}$ is a subgroup of $G$, so we have
\[
l(G) > l(D_{p-1}) = \O(p-1) \geqs n.
\]
This completes the proof.
\hspace *{\fill}$\Box$

\subsection{Proof of Theorem \ref{t:main7}}

We always have $\l(G) < l(G)$ and $h(l) \geqs 36$, so it suffices to show that if $l(G) \geqs 36$ then $\l(G) < h(l)$.

If $G$ is sporadic then $\l(G) \leqs 6$ by Lemma \ref{spor}, and if $G$ is alternating then by Theorem \ref{t:main2} we have $\l(G) \leqs 23$. Hence Theorem \ref{t:main7} holds for these groups.

It remains to deal with groups of Lie type $G = G_r(p^k)$, where $r$ is
the Lie rank of $G$.
Let $B$ be a Borel subgroup of $G$ and let $P < B$ be a Sylow $p$-subgroup of $G$. Note that any unrefinable chain for $G$ passing through $B$ has length $r+l(B)$. Also note that $l(B) = \O(|B|)$ since $B$ is solvable. Define $u(G)$ by $|P| = (p^k)^{u(G)}$.
Then
\[
l(G) \geqs r + l(B) = r + \O(|B|) = r + \O(|B|/|P|) + \O(|P|) \geqs 2r + \O(|P|) = 2r + ku(G)
\]
and thus
\begin{equation}\label{7.1}
k \leqs (l(G)-2r)/u(G).
\end{equation}

We now use Corollary \ref{c:main5}, its notation and its proof.

In the generic case of Theorem \ref{t:main4} we have
\begin{equation}\label{7.2}
\l(G) < 3 \log_2{k} + 36 \leqs 3 \log_2 (l(G)-2) + 36 \leqs h(l(G)),
\end{equation}
where the last inequality is easily checked numerically, using our assumption that $l(G) \geqs 36$.

In case (i) of Theorem \ref{t:main4} we have
\[
\l(G) \leqs \log_2 k + k/\log_2 k +1 \leqs \log_2(l(G)-2) + \frac{l(G)-2}{\log_2(l(G)-2)} + 1 \leqs h(l(G)).
\]
Finally, in case (ii) we have $G = {\rm U}_n(2^k)$ for odd $n \geqs 3$ and for even $k$, say $k=2m$.
We claim that $k \leqs (l(G)-4)/3$ unless $k=2$ and $G = {\rm U}_3(4)$.

Indeed, if the rank $r$ is at least $2$ then this follows from \eqref{7.1}.
So suppose $r=1$. Then $n=3$ and $|B| = ((2^k)^2-1)(2^k)^3$. If $k > 2$ then
$\O((2^k)^2 - 1) = \Omega(2^{4m}-1) \geqs 3$ (since $m \geqs 2$), which yields
$l(G) \geqs 1 + \Omega(|B|) \geqs 1 + 3 + 3k$, proving the claim.
Note that, by \cite[Theorem 1]{ST} we have $l(G) = 1 + \Omega(|B|)$ in this case.

Combining the above claim with Corollary \ref{c:main5}, we conclude that, if $k > 2$, then
\[
\l(G) < 3 \log_2 k + 2k/\log_2(2k) + 35 = f_1(k) \leqs f_1((l(G)-4)/3) \leqs h(l(G)).
\]
Finally, if $k=2$ then $G = {\rm U}_3(4)$ and $l(G) = 9 < 36$, so the result holds
trivially in this case.

This completes the proof of Theorem \ref{t:main7}. \hspace *{\fill}$\Box$

\vs

In fact similar arguments give rise to better bounds.
In the theorem below we adopt the above notation, and let $o(1)$ denote a number
tending to zero as $l(G) \to \infty$.

\begin{thm}\label{t:3.8}
Let $G = G_r(p^k)$. Then
\begin{itemize}\addtolength{\itemsep}{0.2\baselineskip}
\item[{\rm (i)}] $\l(G) < f_1((l(G)-2r)/u(G))$.
\item[{\rm (ii)}] If $r > 1$, then
$\l(G) \leqs \frac{1+o(1)}{r(r+1/2)} \cdot \frac{l(G)}{\log_2 l(G)}$.
\item[{\rm (iii)}] If $G$ is not as in case (i) or (ii) of Theorem \ref{t:main4}, then
$\l(G) \leqs (3+o(1)) \log_2 l(G)$.
\end{itemize}
\end{thm}

\begin{proof}
Part (i) follows immediately from the proof of Theorem \ref{t:main7}, combined with
inequality \eqref{7.1} above.

Part (iii) follows from inequality \eqref{7.2} above.

Finally, part (ii) follows from part (iii) unless $G = {\rm U}_n(2^k)$, with odd $n$ and even $k$.
In the latter case we have $n = 2r+1$ and $u(G) = n(n-1)/2 = r(2r+1)$, so the result
follows from part (i).
\end{proof}

\vs

In fact it may well be that $\l(G) = O(\log_2 l(G))$ for all finite simple groups $G$.
In view of Theorems \ref{t:main2} and \ref{t:3.8} it suffices to prove it for $G$ as in
case (i) or (ii) of Theorem \ref{t:main4}. This depends on better upper bounds on
$\O(2^r-1)$ for $r$ prime, and on $\O(2^{2^a}+1 )$.

It is known that for most natural numbers $n$ we have $\Omega(n) \sim \log \log n$ (see, for instance, 
\cite[Theorem 431]{HW}). It is reasonable to assume -- though impossible to prove using present methods
of Number Theory -- that $2^r - 1$ ($r$ prime) and $2^{2^a}+1$ are less composite than most numbers. In particular we therefore expect that $\O(2^r-1) \leqs \log \log (2^r - 1) \leqs \log r$ for $r \gg 0$,
and that $\O(2^{2^a} + 1) \leqs \log \log (2^{2^a} + 1) \leqs a$ for $a \gg 0$. Note that this implies
that, for primes $r \gg 0$, the largest prime divisor of $2^r-1$ is at least $(2^r-1)^{1/\log r}$, a bound far
stronger than all known bounds, even assuming the ABC conjecture or the Generalized Riemann Hypothesis
(see, for instance, \cite{MP}). Anyway, plugging our two heuristic assumptions into the proof of
Corollary \ref{c:main5} it would follow that $\l(G(p^k)) = O( \log_2 k )$ in all cases,
and this in turn would yield $\l(G) = O(\log_2 l(G))$.

Finally, note that, in view of the lower bound given in Proposition \ref{p:main8}, our above conjectured upper bound on $\l(G)$ in terms of $l(G)$ would be best possible.

\vs

\end{document}